\documentclass[a4paper,12pt]{amsart}

\usepackage{float}
\usepackage{euscript,eufrak,verbatim}
\usepackage{graphicx}
\usepackage{amscd,pifont}
\usepackage[usenames]{color}
\usepackage[colorlinks,linkcolor=red,anchorcolor=blue,citecolor=blue]{hyperref}
\usepackage{bbm}

\numberwithin{equation}{section}

\newtheorem{thm}{Theorem}[section]
\newtheorem{prop}[thm]{Proposition}
\newtheorem{lem}[thm]{Lemma}

\begin{document}
	
	\title{Equivalent definitions of oscillating sequences of higher orders
	}
	
	
	
	
	\author
	{Ruxi Shi}
	\address
	{LAMFA, UMR 7352 CNRS, Universit\'e de Picardie,
		33 rue Saint Leu, 80039 Amiens, France}
	\email{ruxi.shi@u-picardie.fr}
	
	\thanks{}
	\maketitle
	



\begin{abstract}
An oscillating sequence of order $d$ is defined by the linearly disjointness from all $\{e^{2\pi i P(n)} \}_{n=1}^{\infty}$ for all real polynomials $P$ of degree smaller or equal to $d$. A fully oscillating sequence is defined to be an oscillating sequence of all orders. In this paper, we give several equivalent definitions of such sequences in terms of their disjointness from different dynamical systems on tori. 
\end{abstract}

\section{Introduction}

 Oscillating sequences of higher orders and fully oscillating sequences were defined in \cite{f} where it was proved that fully oscillating sequences are orthogonal to all multiple ergodic realizations of topological dynamics of quasi-discrete spectrum in the sense of Hahn-Parry. In \cite{f2}, this orthogonality was proved for all affine maps of zero entropy on any compact abelian group.   
The oscillation of order $1$ was earlier considered in \cite{fj} where Sarnak's conjecture was proved true for a class of dynamics even when the M\"{o}bius function is replaced by an oscillating sequence of order $1$.
In this paper, we give a full discussion on the oscillating sequences of higher order with respect to the disjointness of such a sequence from different dynamical systems
on tori. 

Let us first recall that a sequence of complex numbers $(c_n)$ is said to be \textit{oscillating of order $d$} ($d\ge 1$ being an integer) if 
$$
\lim\limits_{N\to \infty}\frac{1}{N}\sum_{n=0}^{N-1} c_n e^{2\pi i P(n)}=0
$$
for any $P\in \mathbb{R}_d[t]$ where $\mathbb{R}_d[t]$ is the space of real polynomials of degree smaller than or equal to $d$; it is said to be \textit{fully oscillating} if it is oscillating of order $d$ for every integer $d\ge 1$. 



By a dynamical system, we mean a pair $(X,T)$ where $X$ is a compact metric space and $T:X\to X$ is a continuous map. 
Following \cite{s}, we say that a sequence $(c_n)$ is \textit{linearly disjoint} from a dynamical system $(X,T)$ if 
\begin{equation}\label{ort}
\lim\limits_{N\to \infty} \frac{1}{N}\sum_{n=0}^{N-1}c_n f(T^nx)=0,
\end{equation}
for any $x\in X$ and any $f\in C(X)$ where $C(X)$ is the space of continuous functions on $X$. The sequence $(f(T^nx))$ is sometimes called an \textit{observable sequence}, or a \textit{realization} of the dynamical system $(X, T)$.
Usually, (\ref{ort}) is referred to as the orthogonality of the sequence  $(c_n)$ to the realization $(f(T^n x))$.   
Sequences of the form $f_1(T^nx) f_2(T^{2n})\cdots f_\ell(T^{\ell n}x)$ with $f_1, \cdots, f_\ell \in C(X)$ and $x\in X$ are called \textit{multiple ergodic realizations.}   Their orthogonality to fully oscillating sequences was studied in \cite{f,f2}.

An affine map $T$ on a compact abelian group $X$ is of the form 
$$
Tx=Ax+b,
$$
where $b\in X$ and $A$ is an automorphism on $X$. The automorphism $A$ is also called the \textit{section of an automorphism} of the affine $T$.

As observed  in \cite{ls}, an automorphism of zero entropy on $\mathbb{T}^d\times F$, where $\mathbb{T}^d$ is a $d$-dimensional torus and $F$ is a finite abelian group, has the form
$$
T(x,y)=(Ax+By,Cy)
$$
with $A$ an automorphism on $\mathbb{T}^{d}$ of zero entropy, $B$ a morphism from $F$ to $\mathbb{T}^{d}$ and $C$ an automorphism on $F$. Such automorphism is said to be \textit{special} if  
the morphism $N: =A-I$ is nilpotent of order $d$ (i.e. $N^d=0$ but $N^{d-1}\not=0$) and meanwhile the affine map $T$ is also called \textit{special}.

Our main result is as follows. 




\begin{thm}\label{main1}
	A sequence $(c_n)$ of complex numbers  is oscillating of order $d$ if and only if it is linearly disjoint from one of the following dynamical systems or one of the following classes of dynamical systems.
	\begin{itemize}
		\item[(1)]  All automorphisms of zero entropy on $\mathbb{T}^{d+1}\times F$ for any finite abelian group $F$.
		\item[(2)] A special automorphism of zero entropy on $\mathbb{T}^{d+1}\times F$, where $F$ is any fixed finite abelian group $F$.
		\item[(3)] All affine maps of zero entropy on $\mathbb{T}^{d}\times F$ for any finite abelian group $F$.
		\item[(4)] All special affine maps of zero entropy on $\mathbb{T}^{d}\times F$, which share  the same section of an automorphism, where $F$ is any fixed finite abelian group.
	\end{itemize}
\end{thm}

We remark that the finite abelian group $F$ in Theorem \ref{main1} can be the trivial group so that $\mathbb{T}^d \times F$ is nothing but the torus $\mathbb{T}^d$.

For full oscillating sequences, we have the following theorem.
\begin{thm}\label{main2}
	A sequence of complex number $(c_n)$ is full oscillating if and only if it is linearly disjoint from all affine maps on any compact abelian group of zero entropy.
\end{thm}

The necessity in Theorem \ref{main2} is essentially proved in  \cite{ls}, where the authors observe that every orbit of any affine map of zero entropy on any compact abelian group could be realized by an orbit of an affine map of zero entropy on $\mathbb{T}^{d+1}\times F$ for some finite abelian group $F$. On the other hand, the sufficiency is a direct consequence of Theorem \ref{main1}.

A direct consequence from Theorem \ref{main2} is that the $(S)$ sequences which are defined in \cite{akld} to be linearly disjoint from all dynamical systems of zero entropy are fully oscillating. Thus the fully oscillating sequences are candidates of $(S)$ sequences and conversely $(S)$ sequences give abundant examples of fully oscillating sequences.

\medskip

This paper is organized as follows. First, we give some equivalent descriptions of oscillating sequences of high orders in Section \ref{defi}. In Section  \ref{disjointness}, we show the oscillating disjointness from quasi-unipotent maps on compact abelian groups. The key point as observed in \cite{d,ls} is that automorphisms of zero entropy on compact abelian groups are quasi-unipotent.  After these preparations, we prove Theorem \ref{main1} in the last section.


\section{Oscillating property}\label{defi}
In this section, we revisit the definition of oscillating sequences of higher orders. The following lemma is evident by the definition of oscillating sequences because the constant term of the polynomial does not play any role. 


\begin{lem}\label{lem00000}
	Let $d\ge 1$. A sequence of complex numbers $(c_n)$  is oscillating of order $d$
	if and only if 
	$$
	\lim\limits_{N\to \infty}\frac{1}{N}\sum_{n=0}^{N-1} c_n e^{2\pi i nQ(n)}=0
	$$
	for any $Q\in \mathbb{R}_{d-1}[t]$.
\end{lem}

The following lemma has its own interest. A stronger form was proved in \cite{f2}. The proof given here is based on a fact proved in \cite{fj}.

\begin{lem}\label{arithmetic}
	Let $d\ge 1$. 	A sequence of complex numbers $(c_n)$ is oscillating of order $d$ if and only if all arithmetic subsequences $(c_{an+b})$ are oscillating of order $d$ for all integers $a\in \mathbb{N}^*$ and $b\in \mathbb{N}.$ 	
\end{lem}
\begin{proof}
	The sufficiency being trivial, we just prove the necessity. Assume that $(c_n)$ is oscillating of order $d$. By definition, it is easy to see that for any $b\in \mathbb{N}$, the translated sequence  $(c_{n+b})$ is also oscillating of order $d$. Thus we only need to prove that $(c_{an})$ is oscillating of order $d$ for any $a\in \mathbb{N}^*$. We will prove this by induction on $d$ based on Proposition 4 in \cite{fj}, which states that  the case $d=1$ is true.
	

	
	Suppose that the assertion is proved for $d-1$. Take any polynomial $P\in \mathbb{R}_d[t]$ and write $P(t)=wt^d+Q(t)$ with $Q\in \mathbb{R}_{d-1}[t]$.
	Since $(c_n)$ is oscillating of order $d$, the sequence $\widetilde{c}_n:=c_ne^{2\pi i \frac{w}{a^d} n^d}$ is also oscillating of order $d$, a fortiori, $\widetilde{c}_n$ is oscillating of order $d-1$. Then by the hypothesis of induction, $\widetilde{c}_{an}$ is also oscillating of order $d-1$, in particular, 
	$$
	\lim\limits_{N\to \infty}\frac{1}{N}\sum_{n=0}^{N-1} \widetilde{c}_{an} e^{2\pi i Q(n)}=0.
	$$
	Now it suffices to observe
	$$
	\forall n\in \mathbb{N}^*,~~~  \widetilde{c}_{an} e^{2\pi i Q(n)}= c_{an} e^{2\pi i P(n)}
	$$
	to conclude that $(c_{an})$ is oscillating of order $d$.
	
\end{proof}

The following proposition will be used to prove Theorem \ref{main1}.

\begin{prop}\label{rational}
		A sequence of complex numbers $(c_n)$ is oscillating of order $d$ if and only if 
		$$
		\lim\limits_{N\to \infty}\frac{1}{N}\sum_{n=0}^{N-1} c_n e^{2\pi i (P(n)+Q(n))}=0
		$$
		for any $P\in \mathbb{R}_d[t]$ and any $Q\in \mathbb{Q}[t]$.

\end{prop}

\begin{proof}
	The sufficiency being trivial, we just prove the necessity. We first remark that for any $Q\in \mathbb{Q}[t]$, the sequence $(e^{2\pi i Q(n)})$ is periodic. Indeed, 
	on one hand, $(e^{2\pi i \frac{p}{q} n^k})$ is $q$-periodic for any integer $k$ and any rational number $\frac{p}{q}$; on the other hand,
	the product of two periodic sequences is still periodic and the l.c.m of the two periods is a period of the product.

	
	Let $\ell$ be a period of the sequence $(e^{2\pi i Q(n)})$, then
	\begin{align*}
		\sum_{n=0}^{N-1} c_n e^{2\pi i (P(n)+Q(n))}
		&=\sum_{j=0}^{\ell-1}\sum_{\substack {0\le n\le N-1\\n\equiv j \mod \ell}} c_n e^{2\pi i (P(n)+Q(n))}\\
		&=\sum_{j=0}^{\ell-1} e^{2\pi i Q(j)} \sum_{\substack {0\le n\le N-1\\n\equiv j \mod \ell}} c_n e^{2\pi i P(n)}.
	\end{align*}
	Thus we can conclude by Lemma \ref{arithmetic}.
\end{proof}

We finish this section by pointing out that the notion of oscillation of order $d (\ge 2)$ is stronger than the notion of weak oscillation of order $d (\ge 2)$ introduced in \cite{fj}. 
A sequence $(c_n)$ is called \textit{weakly oscillating of order $d$ } if 
$$
\lim\limits_{N\to \infty}\frac{1}{N}\sum_{n=0}^{N-1} c_n e^{2\pi i n^kt}=0
$$
for any $0\le k\le d$ and any $t\in [0,1)$.  A sequence is said to be \textit{fully weakly oscillating} if it is weakly oscillating of order $d$ for all integers $d\ge 1$. 
By Lemma \ref{lem00000}, it is evident that weakly oscillating sequences of order $1$ are exactly oscillating sequences of order $1$. 
But it is not the case when $d \ge 2$. 

\begin{prop}\label{weakoscillating} Let $d\ge 2$. There exist weakly oscillating sequences of order $d$ which are not oscillating of order $d$.
\end{prop}
\begin{proof}
	The following sequence defined by
	$$
	c_n:=e^{2\pi i (n^d\alpha+n^{d-1}\beta)}~~ \text{where}~~ \alpha, \beta\in \mathbb{R}\setminus\mathbb{Q}
	$$
	is weakly oscillating of order $d$ but not oscillating of order $d$.
\end{proof}

A direct corollary of Proposition \ref{weakoscillating} is that a fully weakly oscillating sequence is not necessarily a fully oscillating sequence.

\section{Quasi-unipotent automorphisms}\label{disjointness}

In \cite{ls}, the authors proved the M\"{o}bius disjointness from affine maps on compact abelian group of zero entropy by using the fact that the automorphisms on torus of zero entropy are quasi-unipotnent (\cite{d}). In this section, we examine the orbits of quasi-unipotnent automorphisms on compact abelian groups. Recall that an automorphism $A$ on an abelian group is said to be \textit{quasi-unipotent of type $(m,l)$} if $A^m=I+N$ with $I$ the identity and $N^l=0$, where $m\ge 1$ and $l\ge 1$ are integers.

The following fact is well known. We give a proof of it for the sake of completeness. 
\begin{lem}\label{lema0}
	Let $A$ be an automorphism on $\mathbb{T}^{d}$ of zero entropy. Then $A$ is quasi-unipotent of type $(d,d)$.
\end{lem}
\begin{proof}
The automorphism $A$ can be lifted to a linear automorphism $\widetilde{A}$ of $\mathbb{R}^d$ which preserves $\mathbb{Z}^d$ with $\det\widetilde{A}=\pm 1.$ 
Since the dynamics $(\mathbb{T}^d,T)$ has zero entropy, all the eigenvalues of $\widetilde{A}$ must have absolute value $1$ \cite{si}. According to Kronecker's theorem, all the eigenvalues of $\widetilde{A}$ are $d$-th unit roots. Thus all of eigenvalues of $\widetilde{A}^{d}$ are $1$. By Jordan's theorem,  there exists a matrix $P\in SL(d,\mathbb{C})$ such that $P\widetilde{A}^{d}P^{-1}$ is a triangular matrix with $1$ on its diagonal. Consequently we have $(P\widetilde{A}^{d}P^{-1}-I)^{d}=0$ so that $(\widetilde{A}^{d}-I)^{d}=0$. Therefore, we can write $A^d=I+N$ with  $N^d=0$. 
\end{proof}

Assume that the affine map $Tx= Ax +b$ has its automorphism $A$ which is quasi-unipotent of type $(m, l)$. Then we say that $T$ is quasi-unipotent of type $(m, l)$. Any orbits $(T^n x)$ of such an affine map has a nice behavior as shown in the following lemma. 

\begin{lem}\label{lemma1}
	Let $X$ be a compact abelian group. Let $T$ be an affine map on $X$ defined by $Tx=Ax+b$ where $b\in X$ and $A$ is a quasi-unipotent automorphism of type $(m,l)$. Then for any $n=qm+p\ge l$ where $1\le p\le m$ and $q\in \mathbb{N}$, we have
	$$
	T^nx=\sum_{k=0}^{l-1}(C_q^kN^kA^px+C_{q+1}^{k+1}N^kb^*+C_q^kN^kb_p),
	$$
	where $b^*=\sum_{k=0}^{m-1}A^kb$ and $b_p=\sum_{k=1}^{p}A^kb$. 
\end{lem}
\begin{proof} By the definition of $T$, we have 
\begin{equation}\label{(1)}
T^nx=A^nx+\sum_{j=0}^{n-1}A^jb .
\end{equation}
Let us write 
$$
	\sum_{j=0}^{n-1}A^j
	=\sum_{j=0}^{q-1}\sum_{k=1}^{m}A^{mj+k}+\sum_{k=1}^{p} A^{qm+k}.
$$	
Using the fact $A^m=I+N$, we get
\begin{equation}\label{(2)}
	\sum_{j=0}^{n-1}A^j=\sum_{j=0}^{q-1}(I+N)^j \sum_{k=0}^{m-1}A^k+(I+N)^q\sum_{k=1}^{p}A^k .
\end{equation}	
Recall that $N^l=0$. Now using the interchange of summations, we have
\begin{align*}
	\sum_{j=0}^{q-1}(I+N)^j
	=\sum_{j=0}^{q-1}\sum_{k=0}^{j}C_j^kN^k
	=\sum_{k=0}^{l-1}\sum_{j=k}^{q}C_j^kN^k=\sum_{k=0}^{l-1}N^k\sum_{j=k}^{q}C_j^k.
\end{align*}
By the identity $\sum_{j=k}^{n}C_j^k=C_{n+1}^{k+1}$, we have 
\begin{equation}\label{(3)}
\sum_{j=0}^{q-1}(I+N)^j=\sum_{k=0}^{l-1}C_{q+1}^{k+1}N^k.
\end{equation}

Combining (\ref{(1)}), (\ref{(2)}), (\ref{(3)}) and the equation $(I+N)^q=\sum_{k=0}^{l-1}C^{k}_q N^k$ gives
\begin{align*}
	T^nx
	&=\sum_{k=0}^{l-1}C_q^kN^kA^px+\sum_{k=0}^{l-1}C_{q+1}^{k+1}N^k{b^*}+\sum_{k=0}^{l-1}C_q^kN^kb_p\\
	&=\sum_{k=0}^{l-1}(C_q^kN^kA^px+C_{q+1}^{k+1}N^k{b^*}+C_q^kN^kb_p).
\end{align*}

\end{proof}

Lemma \ref{lemma1} allows us to see that the realization $(\phi(T^n x))$ for a group character $\phi \in \widehat{X}$ is related to exponentials of some real polynomials, where $\widehat{X}$
is the dual group of $X$. 

\begin{lem}\label{lemma2}
	We keep the same assumption as in Lemma  \ref{lemma1}. Let $\phi \in \widehat{X}$, $x\in X$ and $n\ge l$. If we write $n=qm+p$ with $1\le p\le m$ and $q\in \mathbb{N}$, we have
	$$
	\phi(T^nx)=e^{2\pi i P(q)}
	$$
	where $P\in \mathbb{R}_{l}[t]$ is a polynomial depending on $p$. Moreover, if $b$ is of finite order, then $P(t)=P_1(t)+P_2(t)$ with $P_1\in \mathbb{R}_{l-1}[t]$ and $P_2\in \mathbb{Q}[t]$.
\end{lem}
\begin{proof}
	By Lemma \ref{lemma1}, we have
	\begin{align*}
		\phi(T^nx)
		&=\prod_{k=0}^{l-1}\phi(N^kA^px)^{C_q^k}\phi(N^k{b^*})^{C_{q+1}^{k+1}}\phi(N^kb_p)^{C_q^k}\\
		&=\prod_{k=0}^{l-1}e^{2\pi i (C_q^k \beta_{k, p}+ C_{q+1}^{k+1}\alpha_k  +C_q^k \theta_{k,p})}
		=e^{2\pi i (P_1(q)+P_2(q))},
	\end{align*}		
	where 
	$$P_1(q)=\sum_{k=0}^{l-1}C_q^k \beta_{k,p}, \quad P_2(q)=\sum_{k=0}^{l-1}(C_{q+1}^{k+1}\alpha_k + C_q^k\theta_{k,p})$$
	where $\beta_{k,p}$,  $\alpha_k$ and $\theta_{k,p}$ are arguments such that 
	$$ e^{2\pi i \beta_{k,p}}=\phi(N^kA^px),  \quad e^{2\pi i\alpha_k}=\phi(N^k{b^*}), \quad e^{2\pi i  \theta_{k,p}}=\phi(N^kb_p).$$ 
	The polynomial $P_1$ has degree at most $l-1$ and the polynomial $P_2$ has degree at most $l$. 
	
	If $b$ is of finite order so are $b^*$ and $b_p$ for all $1\le p\le m$. Since $N$ is homomorphic, 
	$N^kb^*$ and $N^kb_p$ are all of finite order. This implies that $\alpha_k$ and $\theta_{k,p}$ are rationals so that $P_2\in \mathbb{Q}[t]$. 
\end{proof}


Now we are ready to state and prove the following result on disjointness. 

\begin{prop}\label{keylemma}
	
	 Let $T$ be an affine map on a compact abelian group $X$ defined by $Tx=Ax+b$. 
	Suppose that  $A$ is quasi-unipotent of type $(m,l)$ and that  $(c_n)$ is an oscillating sequence of order $d$. Then
	 $(c_n)$ is linearly disjoint from the dynamical system $(X, T)$ if one of the following conditions is satisfied
	\begin{itemize}
		\item[(1)] $d\ge l,$
		\item[(2)]$d\ge l-1$ and $b$ is of finite order. 
	\end{itemize}
\end{prop}
\begin{proof}
	Since the linear combinations of characters of $X$ are dense in $C(X)$, we have only to check (\ref{ort}) for $f=\phi \in \widehat{X}$.
	Therefore we can conclude by
	combining Lemma \ref{arithmetic}, Proposition \ref{rational}, Lemma \ref{lemma1} and Lemma \ref{lemma2}.
\end{proof}

 We finish this section by two formulas concerning  the powers of a unipotent map $N$ on $\mathbb{T}^d$ which is unipotent of order $d$. By the classification theorem of nilpotent matrices (\cite{bf}, p. 312), under a suitable basis, $N$ has the form
$$
N=
\begin{pmatrix} 
0 & 1 & 0 &\cdots & 0 \\
0 & 0 & 1 &\cdots & 0 \\
0 & 0 & 0 &\cdots & 0 \\
& &\cdots \cdots & &\\
0 & 0 & 0 &\cdots & 1 \\
0 & 0 & 0 &\cdots & 0 \\
\end{pmatrix}.
$$


For $n\ge d-1$, we have
\begin{equation}\label{(4)}
(I+N)^n=\sum_{k=0}^{d-1}C_n^kN^k=
\begin{pmatrix} 
C_n^0 & C_n^1 & C_n^2 &\cdots & C_n^{d-1} \\
0 & C_n^0 & C_n^1 &\cdots & C_n^{d-2} \\
0 & 0 & C_n^0 &\cdots & C_n^{d-3} \\
& &\cdots \cdots & &\\
0 & 0 & 0 &\cdots & C_n^0 \\
\end{pmatrix} .
\end{equation}
By the identity $\sum_{j=k}^{n}C_j^k=C^{k+1}_{n+1}$ and $N^d=0$, we get
$$
	\sum_{k=0}^{n-1}(I+N)^k
	=\sum_{k=0}^{n-1}\sum_{j=0}^{k}C_k^jN^j
	=\sum_{k=0}^{d-1}C_n^{k+1}N^k,
$$
that is to say,	
\begin{equation}\label{(5)}
\sum_{k=0}^{n-1}(I+N)^k=
	\begin{pmatrix} 
		C_n^1 & C_n^2 & C_n^3 &\cdots & C_n^d \\
		0 & C_n^1 & C_n^2 &\cdots & C_n^{d-1} \\
		0 & 0 & C_n^1 &\cdots & C_n^{d-2} \\
		& &\cdots \cdots & &\\
		0 & 0 & 0 &\cdots & C_n^1 \\
	\end{pmatrix} .
\end{equation}

%


These two formulas (\ref{(4)}) and (\ref{(5)}) will be used in the proof of Theorem \ref{main1}.

\section{Proof of Theorem \ref{main1}}\label{proofThm}

We first observe that the implications $(1)\Rightarrow (2)$ and $(3)\Rightarrow (4)$ are trivial. Let us denote by $(0)$  the assertion ``$(c_n)$ is oscillating of order $d$". 

\subsection{Proofs of $(0)\Rightarrow (1)$ and  $(0)\Rightarrow (3)$}   
The proofs of both implications share the same computation, which will give us an explicit expression for the orbit of an affine map of zero entropy. 
Let $T$ be an affine map on $\mathbb{T}^{d}\times F$ of zero entropy. Then $T$ has the form
\begin{equation}\label{T}
T(x,y)=(Ax+By+a,Cy+b)
\end{equation}
with $a\in \mathbb{T}^{d}$, $b\in F$, $A$ an automorphism on $\mathbb{T}^{d}$ of zero entropy, $C$ an automorphism of $F$ and $B$ a morphism from $F$ to $\mathbb{T}^{d}$. This was observed in \cite{ls}.

For any $j\in \mathbb{N}^*$, let $(x_j,y_j):=T^j(x,y)$.   It is easy to check by induction  that
$$
x_j=A^jx+B_jy+a_j, \quad y_j=C^jy+b_j,
$$
for some $a_j\in \mathbb{T}^{d}$, $b_j\in F$ 
and $B_j$ a morphism from $F$ to $\mathbb{T}^{d}$. More precisely, $a_j, b_j$ and $B_j$ satisfy the recursive relations
$$\left\{
\begin{aligned}
a_{j+1} & =  Aa_j+a+Bb_j \\
b_{j+1} & =  b_j+b \\
B_{j+1} & =  AB_j+BC^j.
\end{aligned}
\right.
$$
Let $q=\sharp F !\cdot d$ and $M=A^q-I$ where $\sharp F!$ denotes the factorial of $\sharp F$.  This choice of $q$ makes that\\
\indent (i)  $C^q=I$;\\
\indent (ii) $A^q$ quasi-unipotent of type $(1,d)$, i.e. $A^q = I + M$ with $M^d=0$.\\
In fact, since $C$ is a permutation of $F$, the order of $C$ divides $\sharp F!$. Then $C^{\sharp F!}=I$, a fortiori $C^q=I$. This is (i). Prove now (ii).
Since $A$ is quasi-unipotent of order $(d,d)$ (by Lemma \ref{lema0}), we can write $A^d = I + N$ with $N^d=0$.  Since $d$ divides $q$,  we can write 
$$
M:=A^q -I =(A^d -I) Q =NQ
$$
where $Q$ commute with $N$.   So, $M^d=0$. 

Since $C^q=I$, we have $y_q=y+b_q$. Let 
$$
F_y(x)=A^qx+(B_qy+a_q), \quad G(y)=y+b_q.
$$
Then 
$$
       T^q(x, y) = (F_y(x), G(y)).
$$
Note that $F_y$ is an affine map on $\mathbb{T}^d$ whose  automorphism  $A^q$ is quasi-unipotent  of type $(1,d)$,  and $G$ is
an affine map on $F$ whose automorphism is the identity which is  a quasi-unipotent  automorphism of type $(1,1)$. 
It can be checked, by induction on $n$,  that for  $n\ge 1$,
$$
T^{nq}(x, y) = (F_y^n(x)+H_n, G^n(y)),
$$
where $H_1=0$ and 
\begin{equation}\label{Hn}
H_n=\sum_{i=1}^{n-1}iA^{q(n-i-1)}B_qb_q \ \ \ (\forall n \ge 2)
\end{equation}
are independent of $x$ and $y$.
Then for any $n\ge 2$ and $0\le j\le q-1$, we have $T^{nq+j}(x,y)=T^{nq}(x_j,y_j)$ so that 
\begin{equation}\label{Tpower}
T^{nq+j}(x,y)=(F_{y_j}^n(x_j)+H_n, G^n(y_j)).
\end{equation}

Now let us prove $(0)\Rightarrow (1)$. Assume that $T$ is an automorphism of zero entropy on $\mathbb{T}^{d+1} \times F$. Then $T$ has the form (\ref{T}) with $a=0$, $b=0$ and $A$ a zero entropy automorphism
on $\mathbb{T}^{d+1}$. Hence $a_q=0$ and $b_q=0$ by the above recursive relation.  So  (\ref{Tpower}) takes the form
$$
         T^{nq+j}(x,y)=(F_{y_j}^n(x_j), G^n(y_j)), \ \ \ \mbox{\rm where} \ \  \  F_y(x) = A^q x + B_qy.
$$
The automorphism of $F_y$ is quasi-unipotent of type $(d+1,d+1)$ and the automorphism of $G$ is quasi-unipotent of type $(1,1)$. Therefore the affine map $(F_y(\cdot), G(\cdot))$
on $\mathbb{T}^{d+1}\times F$ is quasi-unipotent of type $(d+1, d+1)$. On the other hand, $B_qy$ is of finite order for any $y\in F$ because $F$ is finite and $B_q$ is a morphism. 
Thus  by Proposition \ref{keylemma} (2) and Lemma \ref{arithmetic}, for any continuous function $f$ on $\mathbb{T}^{d+1}\times F$, we have
\begin{eqnarray*}
   &&  \lim_{N\to \infty} \frac{1}{N} \sum_{n=0}^{N-1} c_{nq+j}f(T^{nq+j}(x,y)) \\
    &=&   \lim_{N\to \infty} \frac{1}{N} \sum_{n=0}^{N-1} c_{nq+j} f(F_{y_j}^n(x_j), G^n(y_j)) =0. 
\end{eqnarray*}
Taking average over $0\le j<q$, we finish the proof of 
 the implication $(0)\Rightarrow (1)$.

It remains to prove $(0)\Rightarrow (3)$. Assume that $T$ is an affine map of zero entropy on $\mathbb{T}^{d} \times F$. Then $T$ has the form (\ref{T}) and $T^{nq+j}$ takes the form (\ref{Tpower}). 
The automorphism of $G$ is quasi-unipotent of type $(1,1)$. The automorphism of $F_y$ is quasi-unipotent of type $(d,d)$ but the term $a_q$ may not be of finite order (this is the crucial difference from the implication $(0)\Rightarrow (1)$). 
We first show that $n \mapsto H_n$ can be expressed as a polynomial of $n$, namely  
 \begin{equation}\label{Hn2}
 H_n=\sum_{k=0}^{d-1} C_{n}^{k+2} M^k B_qb_q.
 \end{equation}
In fact, by the identity $\sum_{j=k}^{n-2}(n-j-1)C_j^k=C_{n}^{k+2}$ and the fact
$A^q = I + M$ with $M^d=0$, we have 
\begin{align*}
	\sum_{j=0}^{n-2}(n-j-1)A^{qj}
	&=\sum_{j=0}^{n-2}(n-j-1)(I+M)^{j}\\
	&=\sum_{j=0}^{n-2}\sum_{k=0}^{j}(n-j-1)C_j^kM^k\\
	&=\sum_{k=0}^{d-1}\sum_{j=k}^{n-2}(n-j-1)C_j^kM^k\\
	&=\sum_{k=0}^{d-1} C_{n}^{k+2} M^k.\\
\end{align*}
Thus (\ref{Hn2}) follows immediately from (\ref{Hn}). Since $F$ is finite and $B_q$ is a morphism,
 $M^k B_qb_q$ are of finite order for all $0\le k\le d-1.$ Take an arbitrary character $\phi=\phi_1\phi_2$ of $\mathbb{T}^{d} \times F$ where $\phi_1$  is a character of $\mathbb{T}^{d}$ and $\phi_2$ is a character of  $F$.
 Assume $\phi_1(x) = e^{2\pi i (\lambda_1x_1 + \cdots +\lambda_dx_d)}$ where $(\lambda_1, \cdots, \lambda_d)\in \mathbb{Z}^d$.
  By the above discussion, we have $\phi_1(H_n)=e^{2\pi i P(n)}$ where
   $$ P(t) = \sum_{k=0}^{d-1} C_{t}^{k+2} \alpha_k  \in\mathbb{Q}[t]$$
   with $\alpha_k\in \mathbb{Q}$ defined by $e^{2\pi i \alpha_k}=\phi_1(M^kB_qb_q).$
    Then, by (\ref{Tpower}), we have
  \begin{eqnarray*}
   &&  \lim_{N\to \infty} \frac{1}{N} \sum_{n=0}^{N-1} c_{nq+j}\phi(T^{nq+j}(x,y)) \\
    &=&   \lim_{N\to \infty} \frac{1}{N} \sum_{n=0}^{N-1} c_{nq+j}e^{2\pi i P(n)} \phi_1(F_{y_j}^n(x_j)) \phi_2( G^n(y_j)) =0. 
\end{eqnarray*}
The last equality follows from  Lemma \ref{arithmetic},  Proposition \ref{rational} and Proposition \ref{keylemma} (1) which is applied to the affine map
$(x, y) \mapsto (F_{y_j}(x), G(y))$.
By taking average over $0\le j<q$, we complete the proof of 
 the implication $(0)\Rightarrow (3)$.

\subsection{Proof of $(2)\Rightarrow (0)$}
Let $T$ be a special automorphism of zero entropy on $\mathbb{T}^{d+1}\times F$ where $F$ is a finite abelian group. It  has the form
$$
T(x,y)=(Ax+By,Cy)
$$
with $A=I+N$ an automorphism on $\mathbb{T}^{d+1}$ of zero entropy and $N$ of the nilpotent index $d+1$, $C$ an automorphism of $F$ and $B$ a morphism from $F$ to $\mathbb{T}^{d+1}$.  What we will prove  is that any polynomial exponential sequence $(e^{2\pi i P(n)})$ with $P\in \mathbb{R}_d[t]$ is an observable sequence of the system $(\mathbb{T}^{d+1}\times F, T)$.

 Observe that for any $n\in \mathbb{N}^*$
$$
T^{n}(x,0)=((I+N)^nx,0).
$$
Let $P\in \mathbb{R}_d[t]$. Since the binomial polynomials $C_t^0, C_t^1, \dots, C_t^d$ form a basis of $\mathbb{R}_d[t]$, there exists $x:=(x_0, x_1,\dots, x_d)\in \mathbb{R}^{d+1}$ such that
$$
P(t)= x_0C_t^0+x_1C_t^1+\cdots+x_dC_t^d.
$$
Let $G:\mathbb{T}^{d+1}\to \mathbb{T}$ be the projection on the first coordinate. By the formula (\ref{(4)}),  we have $$
e^{2 \pi iP(n)}=e^{2\pi iG(T^nx)}
$$
where $x$ also denotes the projection of $x \in \mathbb{R}^{d+1}$ onto $\mathbb{T}^{d+1}$. So, $(e^{2 \pi iP(n)})$ is an
 observable sequence of the system $(\mathbb{T}^{d+1}\times F,T)$. This completes the proof  $(2)\Rightarrow (0)$.

\subsection{Proof of $(4)\Rightarrow (0)$} 
Let $T_b$ be a special affine map of zero entropy on $\mathbb{T}^{d}\times F$ for a finite abelian group $F$ which has the form
$$
T_b(x,y)=(Ax+By+b,Cy)
$$
with $b\in \mathbb{T}^{d}$, $A=I+N$ an automorphism on $\mathbb{T}^{d}$ of zero entropy and $N$ of the nilpotent index $d$, $C$ an automorphism of $F$ and $B$ a morphism from $F$ to $\mathbb{T}^{d}$. 
We will prove that  any polynomial exponential sequence $(e^{2\pi i P(n)})$ with $P\in \mathbb{R}_d[t]$ is an observable sequence of the system $(\mathbb{T}^{d}\times F,T_b)$ along the orbit of 
$(0, 0) \in \mathbb{T}^{d}\times F$.

Observe that 
$$
T_b^n(0,0)=\left(\sum_{k=0}^{n-1}(I+N)^kb,0\right).
$$
Let $Q\in \mathbb{R}_{d-1}[t]$. Since $1, C_t^1, \dots, C_t^d$ form a basis of $\mathbb{R}_d[t]$, there exists $b=(b_1, b_2,\dots, b_d)\in \mathbb{R}^{d}$ and $b_0\in \mathbb{R}$ such that
$$
t Q(t)=b_0+b_1C_t^1+\cdots+b_dC_t^d.
$$
Let $G:\mathbb{T}^{d}\times F\to \mathbb{T}$ be the projection on the first coordinate of $\mathbb{T}^{d}$. By the formula (\ref{(5)}),  we get 
$$
e^{2 \pi inQ(n)}=e^{2\pi i(G(T_b^n(0,0))+b_0)}
$$
which is an observable sequence of the system $(\mathbb{T}^{d}\times F,T_b)$. By Lemma \ref{lem00000}, this completes the proof  $(4)\Rightarrow (0)$.

\section{Acknowledgment}
The authors are grateful to the anonymous reviewers for their valuable remarks.

\end{document}